\documentclass[a4paper, 12pt, reqno]{amsart}
\numberwithin{equation}{section}

\usepackage{amssymb}
\usepackage{mathrsfs}
\usepackage{dsfont}
\usepackage{stmaryrd} %, MnSymbol}
\usepackage{enumerate, xspace}
\usepackage{changebar}
\usepackage{hyperref}
\usepackage{pdfsync}

\newtheorem{theorem}{Theorem}[section]
\newtheorem{lemma}{Lemma}[section]

\newtheorem{proposition}{Proposition}[section]

\newtheorem{remark}{Remark}[section]

\newtheorem{Ass}{Assumption}[section]
\numberwithin{equation}{section}

\renewcommand{\a}{\alpha}
\renewcommand{\b}{\beta}

\newcommand{\la}{\lambda}

\newcommand{\La}{\Lambda}

\def\R{{\mathbb{R}}}
\def\N{{\mathbb{N}}}

\def\D{{\mathcal{D}}}

\def\F{{\mathcal{F}}}
\def\L{{\mathcal{L}}}
\allowdisplaybreaks

\begin{document}
\title[SG of binary collision processes on $d$-dim. lattice ]
{On the spectral gap of the Kac walk and other binary collision processes on $d$-dimensional lattice}
\author{Makiko Sasada$^{*)}$}
%\address{Makiko Sasada\\
%Department of Mathematics, Keio University\\
%3-14-1, Hiyoshi, Kohoku-ku, Yokohama-shi, Kanagawa, 223-8522, Japan}
%\email{{\tt sasada@math.keio.ac.jp}}
%\date{August 23, 2013}
\footnote{
\hskip -6mm
$^{*)}$ Department of Mathematics, Keio University,
3-14-1, Hiyoshi, Kohoku-ku, Yokohama 223-8522, Japan.
e-mail: sasada@math.keio.ac.jp 
}
%\footnote{
%\hskip -6mm
%\textit{Abbreviated title $($running head$)$: SG of binary collision processes on $d$-dim. lattice .}}

\footnote{
\hskip -6mm
\textit{The work wad supported by JSPS Grant-in-Aid for Research Activity Start-up Grant Number 23840036.
}}

\begin{abstract}
We give a lower bound on the spectral gap for a class of binary collision processes. In \cite{Ca08}, Caputo showed that, for a class of binary collision processes given by simple averages on the complete graph, the analysis of the spectral gap of an $N$-component system is reduced to that of the same system for $N = 3$. In this paper, we give a comparison technique to reduce the analysis of the spectral gap of  binary collision processes given by simple averages on $d$-dimensional lattice to that on the complete graph. We also give a comparison technique to reduce the analysis of the spectral gap of binary collision processes which are not given by simple averages to that given by simple averages. Combining them with Caputo's result, we give a new and elementary method to obtain spectral gap estimates. The method applies to a number of binary collision processes on the complete graph and also on $d$-dimensional lattice, including a class of energy exchange models which was recently introduced in \cite{GKS}, and zero-range processes. 

\end{abstract}
%\thanks{ This paper has been partially supported by the
%  European Advanced Grant {\em Macroscopic Laws and Dynamical Systems}
%  (MALADY) (ERC AdG 246953), by Agence Nationale de la Recherche, 
%  under grant ANR-2010-BLAN-0108 (SHEPI).\\
%  We thank Professor T. Funaki for insightful discussions and his interest in this work.} 
%\keywords{}
%\subjclass[2000]{}
%%%%%%    TEXT START    %%%%%%

\maketitle

\section{Introduction}\label{intro}

A sharp lower bound on the spectral gap of the process is essential to prove the hydrodynamic limit (cf. \cite{KL}). What is needed is that the gap, for the process confined to cubes of linear size $N$, shrinks at a rate $N^{-2}$. Up to constants, this is the best possible lower bound for a wide class of models discussed in the context of the study of the hydrodynamic limit. 

Most of the techniques used to obtain the required lower bound rely on special features of the model, or a recursive approach due to Lu and Yau \cite{LY}.
Recently, Caputo introduced a new and elementary method to obtain a lower bound on the spectral gap for some general class of binary collision processes which are reversible with respect to a family of product measures in \cite{Ca08}. In this paper, we extend his result in two ways. One way is that though in \cite{Ca08} only the process on the complete graph was 
considered, we consider the process on $d$-dimensional lattice where the interactions occur between nearest neighbor sites. We give a general method to compare the spectral gap of the local version on $d$-dimensional lattice and the original process on the complete graph. Secondly, we study a wider class of processes than the class studied in \cite{Ca08} and give a simple comparison technique between their spectral gaps. We emphasize that our technique can be applied to a wide class of processes which are reversible with respect to a family of product measures, and it allows to obtain the lower bound of the spectral gap easily. However, it is not necessarily sharp, so if the estimate given by our method is not enough sharp, then we need to try to use other techniques.

Following Caputo \cite{Ca08}, we first consider the following energy conserving binary collision model introduced by M. Kac in \cite{K}, called Kac walk. Let $\nu=\nu_{N,\omega}$ denote the uniform probability measure on the $N-1$ dimensional sphere of the radius $\sqrt{\omega}$
\begin{equation*}
S^{N-1}(\omega)=\{ \eta \in \R^N; \sum_{i=1}^N \eta_i^2 =\omega \},
\end{equation*}
and consider the $\nu$-reversible Markov process on $S^{N-1}(\omega)$ with infinitesimal generator given by 
\begin{equation*}
\mathcal{L^*}f(\eta)=\frac{1}{2N}\sum_{i,j=1}^N D_{i,j}f(\eta)
\end{equation*}
where 
\begin{equation*}
D_{i,j}f(\eta)= \frac{1}{2\pi} \int_{-\pi}^{\pi} [ f(R^{ij}_{\theta}\eta)-f(\eta)]d\theta,
\end{equation*}
and $R^{ij}_{\theta}, i \neq j$ is a clockwise rotation of angle $\theta$ in the plane $(\eta_i,\eta_j)$. As a convention, we take $R^{ii}_{\theta}=Id$.

This Kac walk represents a system of $N$ particles in one dimension evolving under a random collision mechanism. The state of the system is given by specifying the $N$ velocities $\eta_1, \eta_2, \cdots, \eta_N$. The random collision mechanism under which the state evolves is that at random times, a \lq\lq pair collision" take place in such a way that the total energy $\sum_{i=1}^N \eta_i^2$ is conserved. Under the above dynamics, after the particles $i$ and $j$ collide, the distribution of the velocities $(\eta_i,\eta_j)$ becomes uniform on the plane $(\eta_i,\eta_j)$.

Note that $-\mathcal{L^*}$ is a non-negative, bounded self-adjoint operator on $L^2(\nu)$. Any constant is an eigenfunction with eigenvalue $0$ and the spectral gap $\lambda^*=\lambda^*(N,\omega)$ is defined as 
\begin{equation}\label{eq:sg}
\la^*(N, \omega):=\inf\Big\{ \frac{ \nu(f(\mathcal{-L^*})f)}{\nu(f^2)}  \Big| \nu(f)=0, \ f \in L^2(\nu) \Big\}
\end{equation}
where $\nu(f)$ stands for the expectation $\int f d\nu$. We define $\la^*(N)=\displaystyle \inf_{\omega>0}\la^*(N, \omega)$. For the Kac walk, by change of variables, it is easy to see that $\la^*(N)=\la^*(N, \omega)$ for all $\omega>0$.

In \cite{CCL}, Carlen, Carvalho and Loss computed the exact value of $\la^*(N)$ for every $N$:
\begin{equation}
\la^*(N)=\frac{N+2}{4N}, \quad N \ge 2. \label{eq:exactvalue}
\end{equation}
Caputo gave a simplified method to show this. Recall Theorem 1.1 in \cite{Ca08}.
\begin{theorem}[Caputo]\label{thm:caputo0}
For $N \ge 3$,
  \begin{equation}\label{eq:caputo0}
\la^*(N) = (3\la^*(3)-1)(1-\frac{2}{N})+\frac{1}{N}.
  \end{equation}
In particular, (\ref{eq:exactvalue}) follows from (\ref{eq:caputo0}) with $\la^*(3)=\frac{5}{12}$.
\end{theorem}

Now, we introduce the local version of the Kac walk.
Fix $d \in \N$ and let $\La_N$ the $d$-dimensional cube of linear size $N$ : $\La_N=\{1,2, \cdots ,N\}^d$.
The local version of the Kac walk is the $\nu=\nu_{|\La_N|,\omega}=\nu_{N^d,\omega}$-reversible Markov process on $S^{|\La_N|-1}(\omega)$ with infinitesimal generator given by 
\begin{equation}\label{eq:genenn}
\mathcal{L}^{*,loc}f(\eta)=\frac{1}{2}\sum_{x \in \La_N} \sum_{\substack{y \in \La_N \\ \|x-y\|=1}}D_{x,y}f(\eta)
\end{equation}
where $\|x-y\|=\sum_{i=1}^d |x_i - y_i|$.
We define the spectral gap $\lambda^{*,loc}(N,\omega)$ by (\ref{eq:sg}) with $-\mathcal{L}^*$ replaced by $-\mathcal{L}^{*,loc}$, and $\la^{*,loc}(N):=\displaystyle \inf_{\omega>0}\la^{*,loc}(N, \omega)$. It is also easy to see that $\la^{*,loc}(N)=\la^{*,loc}(N, \omega)$ for all $\omega>0$.

We give a comparison theorem for $\la^{*,loc}(N)$ and $\la^{*}(N)$.
\begin{theorem}\label{thm:nn}
  \begin{equation*}
\la^{*,loc}(N) \ge \frac{1}{96 d N^2}\la^*(|\La_N|).
  \end{equation*}
In particular, since $\la^*(|\La_N|) \ge \frac{1}{4}$ for all $N \ge 2$ by (\ref{eq:exactvalue}),  
\begin{equation}\label{eq:nnbound}
\la^{*,loc}(N) \ge \frac{1}{ 384d N^2 }. 
\end{equation}
%In particular, $\la^{*,loc}(N) \ge O(N^{-2})$.
\end{theorem}

In the proof, we use the invariance of $\nu$ under the exchange of coordinates repeatedly, and the idea of \lq\lq moving particle lemma" which was developed for the study of the spectral gap of interacting particle systems with discrete spins (cf. \cite{NS}).
Generally, it is not easy to show the estimate corresponding to \lq\lq moving particle lemma" for the systems with continuous spins. However, for the generator of the form (\ref{eq:genenn}), 
we show that the estimate can be established. A proof of Theorem \ref{thm:nn} is given in the next subsection.

Next, we consider a generalization of Kac walk introduced in \cite{CCL} by Carlen et al.
Let $\rho(\theta)$ be a probability density on the circle, i.e.
 \begin{equation*}
\int_{-\pi}^{\pi} \rho(\theta)=1.
  \end{equation*}
Consider a $\nu=\nu_{N,\omega}$-reversible Markov process on $S^{N-1}(\omega)$ with infinitesimal generator given by 
\begin{equation*}
\mathcal{L}f(\eta)=\frac{1}{2N}\sum_{i,j=1}^N  \int_{-\pi}^{\pi} [ f(R^{ij}_{\theta}\eta)-f(\eta)] \rho(\theta) d\theta.
\end{equation*}
We define the spectral gap $\lambda(N,\omega)$ by (\ref{eq:sg}) with $-\mathcal{L}^*$ replaced by $-\mathcal{L}$ and $\la(N):=\displaystyle \inf_{\omega>0}\la(N, \omega)$. For this generalization, $\la(N)=\la(N, \omega)$ holds for any $\omega>0$ again, since $\mathcal{L}$ commutes with the unitary change of scale from $S^{N-1}(\omega)$ to $S^{N-1}(\omega')$ , for any $\omega, \omega' >0$. Indeed, $\nu_{N,\omega'}$ is the image of $\nu_{N,\omega}$ under the map $T : \eta \to \frac{\sqrt{\omega'} \eta}{\sqrt{\omega}}$ and if $f_{T}(\eta)=f(T\eta)$, then
\begin{equation}
\nu_{N,\omega} (f_T (-\mathcal{L})f_T)= \nu_{N,\omega'} (f (-\mathcal{L})f)
\end{equation}
holds.
Note that, to guarantee $\la(N) > 0$, we need some more assumptions on $\rho$.

We introduce the local version of this generalized Kac walk described by the infinitesimal generator 
\begin{equation*}
\mathcal{L}^{loc}f(\eta)=\frac{1}{2}\sum_{x \in \La_N} \sum_{\substack{y \in \La_N \\ \|x-y\|=1}} \int_{-\pi}^{\pi} [ f(R^{xy}_{\theta}\eta)-f(\eta)] \rho(\theta) d\theta
\end{equation*}
and define the spectral gap $\lambda^{loc}(N,\omega)$ and $\lambda^{loc}(N)$, which satisfying $\lambda^{loc}(N)=\lambda^{loc}(N,\omega)$ for all $\omega>0$, in the same manner as before.
In \cite{CCL}, under the assumption that $\rho(\theta)$ is continuous and $\rho(0) >0$, it is shown that 
\begin{equation}
\la(N) \ge \la(2)\frac{N+2}{2N}, \quad N \ge 2. \label{eq:boundgeneral}
\end{equation}
Under their assumption on $\rho(\theta)$, it is also proved that $\la(2)>0$ and therefore $\la(N) >0$.

Our next result shows that the proof of (\ref{eq:boundgeneral}) can be somewhat simplified, and we also have a lower bound on $\lambda^{loc}(N)$. Note that since we only assume that $\rho(\theta)$ is a probability density on the circle, $\la(2)$ is not necessarily positive.
\begin{theorem}\label{thm:comparison}
%Assume that our generator (\ref{eq:generator}) is reversible with respect to $\nu$.
%Then,  
  \begin{equation*}
\la(N) \ge 2\la(2) \la^*(N),  \quad \la^{loc}(N) \ge 2\la(2) \la^{*,loc}(N).
  \end{equation*}
In particular, with (\ref{eq:exactvalue}), we have (\ref{eq:boundgeneral}) and with (\ref{eq:nnbound}), we have 
\begin{equation*}
\la^{loc}(N) \ge \la(2) \frac{1}{192dN^2}.
\end{equation*}
\end{theorem}

In the next result, we also give an upper bound of $\lambda(N)$ and $\lambda^{loc}(N)$. Denote the supremum of the spectral of $-\mathcal{L}$ for $N=2$ by $\kappa$:
\begin{equation*}
\kappa:=\sup_{\omega>0} \sup\Big\{ \frac{ \nu_{2,\omega}(f(-\mathcal{L})f)}{\nu_{2,\omega}(f^2)}  \Big| \nu_{2,\omega}(f)=0, \ f \in L^2(\nu_{2,\omega}) \Big\}.
\end{equation*}
\begin{theorem}\label{thm:comparisonupper}
%Assume that our generator (\ref{eq:generator}) is reversible with respect to $\nu$.
%Then,  
  \begin{equation*}
\la(N) \le 2 \kappa \la^*(N),  \quad \la^{loc}(N) \le  2 \kappa \la^{*,loc}(N).
  \end{equation*}
In particular, since $\kappa \le 1$, we have $\la(N) \le 2 \la^*(N)$ and $\la^{loc}(N) \le  2 \la^{*,loc}(N)$.
\end{theorem}

The key of the proofs of the above comparison theorems is the fact that $\nu((D_{i,j}f)^2)$ is the expectation of the variance of $f$ with respect to $\nu(\cdot | \F_{i,j})$ where $\F_{i,j}$ is the sigma algebra generated by the coordinates $\{ \eta_k ; k \neq i,j \}$, and therefore this variance can be estimated by the term of $\la(2)$ or $\kappa$ and the corresponding Dirichlet form. Proofs of Theorem \ref{thm:comparison} and Theorem \ref{thm:comparisonupper} are given in the subsection 1.2. 

In section 2, we shall show that variants of the same methods can be used to obtain spectral gap estimates for several models sharing some of the features of the Kac walk or the generalization of Kac walk.
In section 3, we give two examples of such processes.

\subsection{Proof of Theorem \ref{thm:nn}} 

We first introduce operators $E_{i,j}$ appearing in the definition of $\mathcal{L}^*$ and $\pi_{i,j}$ which represents the exchange of the velocity of particles $i$ and $j$:
\begin{equation*}  
 E_{i,j}f (\eta) =\frac{1}{2\pi} \int^{\pi}_{-\pi} f(R^{ij}_{\theta} \eta) d \theta,
\quad 
 \pi_{i,j}f(\eta)=f(\pi_{i,j}\eta)
 \end{equation*}
where
\begin{equation*}  
 (\pi_{i,j} \eta )_k = \begin{cases}
	\eta_k  & \text{if $k \neq i,j$,} \\
	\eta_j & \text{if $k=i$,} \\
	\eta_i  & \text{if $k=j.$} 
\end{cases}
\end{equation*}
As a convention, we take $\pi_{i,i}\eta=\eta$. Note that $E_{i,j}$ is a projection which coincides with $\nu$-conditional expectation given $\sigma$-algebra $\mathcal{F}_{i,j}$ generated by variables $\{ \eta_k ; k \neq i,j\}$. In other words, $E_{i,j}f=\nu(f| \F_{i,j})$ is an average of $f$ on the $(\eta_i,\eta_j)$ plane with respect to $\nu$. Therefore, we regard this model as a binary collision process given by simple averages. Note that, by the definition $D_{i,j}=E_{i,j}-\text{Id}$.

To compare the Dirichlet form with respect to the long range operators with that of the local operators, we first prepare preliminary lemmas.
\begin{lemma}\label{lem:longrange}
For any $x, y$ and $z \in \La_N$ satisfying $y \neq z$,
\[
\nu((D_{x,y}f)^2) \le 6 \nu((\pi_{x,z}f-f)^2)+ 3\nu((D_{z,y}f)^2)
\]
for all $f \in L^2(\nu)$.
\end{lemma}
\begin{proof}
If $x = y$, then $ \nu((D_{x,y}f)^2)=0$, so the inequality obviously holds.
On the other hand, if $x \neq y$ and $y \neq z$, then for any $\eta$,
\[
E_{x,y}f(\eta)=\pi_{x,z} (E_{z,y} (\pi_{x,z}f)) (\eta) = (E_{z,y} (\pi_{x,z}f)) (\pi_{x,z} \eta) .
\]
Therefore, by Schwarz's inequality and change of variables, we have
\begin{align*}
\nu & ((D_{x,y}f)^2)  =\nu((E_{x,y}f-f)^2) = \nu (\{(E_{z,y} (\pi_{x,z}f)) ( \pi_{x,z}\eta) -f (\eta) \}^2) \\
& = \nu (\{(E_{z,y} (\pi_{x,z}f)) (\eta) - (E_{z,y}f) ( \eta) + (E_{z,y}f) ( \eta) -f( \eta ) + f (\eta) - f (\pi_{x,z}\eta) \}^2) \\
& \le 3 \nu (\{(E_{z,y} (\pi_{x,z}f)) (\eta) -(E_{z,y}f) (\eta) \}^2) +3\nu((D_{z,y}f)^2)+ 3\nu((\pi_{x,z}f-f)^2). 
\end{align*}
Finally, since $\nu (\{(E_{z,y} (\pi_{x,z}f)) (\eta) -(E_{z,y}f) (\eta) \}^2)= \nu (\{E_{z,y} (\pi_{x,z}f-f) \}^2) \le \nu(E_{z,y} (\pi_{x,z}f-f)^2)= \nu((\pi_{x,z}f-f)^2 )$, we complete the proof.
\end{proof}

\begin{lemma}\label{lem:exchange}
For any $x, y \in \La_N$,
\[
\nu((\pi_{x,y}f-f)^2) \le 4 \nu((D_{x,y}f)^2).
\]
\end{lemma}
\begin{proof}
Since $E_{x,y}f(\eta)=E_{x,y}f(\pi_{x,y}\eta)$, by Schwarz's inequality, we have
\[
\nu((\pi_{x,y}f-f)^2) = \nu((\pi_{x,y}f-E_{x,y}f+E_{x,y}f -f)^2) \le 4 \nu((D_{x,y}f)^2).
\]
\end{proof}

%\begin{theorem}\label{thm:nearestneighbor}
%  \begin{equation}
%\la^{*,loc}(N) \ge \frac{1}{96 dN^2 }\la^*(|\La_N|).
%  \end{equation}
%\end{theorem}
\begin{proof}[Proof of Theorem \ref{thm:nn}]
For each pair $x, y \in \La_N \ (x \neq y)$, choose a canonical path $\Gamma(x,y)=(x=z_0,z_1, \cdots, z_{n(x,y)}=y)$ where $n(x,y) \in \N$ and $\|z_i-z_{i+1}\|=1$ for $0 \le i \le n(x,y)-1$ by moving first in the first coordinate direction, then in the second coordinate direction, and so on.
Then, by Lemma \ref{lem:longrange}, we have
\begin{equation}\label{eq:1}
\nu((D_{x,y}f)^2) \le 6 \nu((\pi_{x,z_{n(x,y)-1}}f-f)^2)+ 3\nu((D_{z_{n(x,y)-1,y}}f)^2).
\end{equation}

On the other hand, since 
\begin{align*}
 & \pi_{x,z_{n(x,y)-1}}  = \\
& \pi_{z_0,z_1} \circ \pi_{z_1,z_2} \circ \cdots \pi_{z_{n(x,y)-3},z_{n(x,y)-2}} \circ \pi_{z_{n(x,y)-2},z_{n(x,y)-1}} \circ \pi_{z_{n(x,y)-3},z_{n(x,y)-2}} 
 \cdots \circ \pi_{z_1,z_2}  \circ \pi_{z_0,z_1},
\end{align*}
by Schwarz's inequality
\begin{equation}\label{eq:2}
\nu((\pi_{x,z_{n(x,y)-1}}f-f)^2) \le  4 n(x,y) \sum_{i=0}^{n(x,y)-2} \nu((\pi_{z_i,z_{i+1}}f-f)^2).
\end{equation}
Therefore, combining the inequalities (\ref{eq:1}), (\ref{eq:2}) and Lemma \ref{lem:exchange}, we have
\[
\nu((D_{x,y}f)^2) \le 96 \  n(x,y) \sum_{i=0}^{n(x,y)-1}  \nu((D_{z_i,z_{i+1}}f)^2).
\]
Then, by the construction of canonical paths,
\begin{align*}
\nu(f (-\mathcal{L}^*)f) &= \frac{1}{|\La_N|}\sum_{x,y \in \La_N}\nu((D_{x,y}f)^2) \le 96 d N \frac{1}{|\La_N|}\sum_{x,y \in \La_N}  \sum_{i=0}^{n(x,y)-1}  \nu((D_{z_i,z_{i+1}}f)^2) \\
& \le 96dN^2  \sum_{\substack{x,y \in \La_N \\ \|x-y\|=1}} \nu((D_{x,y}f)^2) = 96dN^2 \nu(f (-\mathcal{L}^{*,loc})f ).
\end{align*}

\end{proof}

\begin{remark}
The key ideas of the proof of Theorem \ref{thm:nn}, Lemma \ref{lem:longrange} and \ref{lem:exchange} were exactly same as the ideas presented in Section 2.5 of
\cite{CCM}.
\end{remark}

\subsection{Proof of Theorem \ref{thm:comparison} and Theorem \ref{thm:comparisonupper}}
We define an operator $\mathcal{L}_0$ on $L^2(\nu_{2,\omega})$ as
\begin{equation*}
\mathcal{L}_0 f (\eta)= \frac{1}{2} \{ \int_{-\pi}^{\pi} [ f(R^{12}_{\theta}\eta)-f(\eta)] \rho(\theta) d\theta  + \int_{-\pi}^{\pi} [ f(R^{21}_{\theta}\eta)-f(\eta)] \rho(\theta) d\theta \}
\end{equation*}
where $\eta \in \R^2$. For $N \ge 3$, $\eta \in \R^N$, $1 \le i <j \le N$ and $f: \R^N \to \R$, define $f^{i,j}_{\eta} : \R^2 \to \R$ as
\begin{equation*}
f^{i,j}_{\eta}(p,q)=f(\eta_1,\eta_2, \cdots, \eta_{i-1}, p, \eta_{i+1}, \cdots, \eta_{j-1}, q, \eta_{j+1}, \cdots, \eta_N).
\end{equation*}
Then, we can rewrite the Markov generator as follows:
\begin{equation*}
\mathcal{L}f(\eta)=\frac{1}{N}\sum_{i < j}\mathcal{L}_{i,j} f(\eta)
\end{equation*}
where $\mathcal{L}_{i,j} f (\eta) = (\mathcal{L}_0 f^{i,j}_{\eta} ) (\eta_i,\eta_j)$. Note that $f^{i,j}_{\eta}$ does not depend on $\eta_i$ and $\eta_j$. Then, we have
\begin{equation}\label{eq:2-Nrelation}
\nu(f (-\mathcal{L}_{i,j})  f) = \nu(  f^{i,j}_{\eta} (\eta_i, \eta_j)  ((-\mathcal{L}_0) f^{i,j}_{\eta})(\eta_i, \eta_j) )  = \nu( \nu_{2, \eta_i^2+\eta_j^2}( f^{i,j}_{\eta} (-\mathcal{L}_0) f^{i,j}_{\eta}) ) .
\end{equation}
Note that for $N=2$, $\mathcal{L}=\frac{1}{2}\mathcal{L}_0$. Therefore, by definition, we have for any $\omega > 0$ and $g \in L^2(\nu_{2,\omega})$,
\[
2 \lambda(2) \nu_{2, \omega}( \{  g - \nu_{2, \omega} (g) \} ^2 ) \le \nu_{2, \omega}(  g (-\mathcal{L}_0) g  ) \le 2 \kappa \nu_{2, \omega}( \{  g - \nu_{2, \omega} (g) \} ^2).
\]
Since
\[
\nu(\nu_{2, \eta_i^2+\eta_j^2}( \{  f^{i,j}_{\eta} - \nu_{2, \eta_i^2+\eta_j^2} (f^{i,j}_{\eta}) \} ^2 )) = \nu( \{  f^{i,j}_{\eta} - \nu_{2, \eta_i^2+\eta_j^2} (f^{i,j}_{\eta}) \} ^2 )= \nu(\{f-E_{i,j}f\}^2),
\]
we have
\begin{equation*}
  2\lambda(2) \nu ((D_{i,j}f)^2 ) \le  \nu(f (-\mathcal{L}_{i,j})  f) \le 2\kappa \nu ((D_{i,j}f)^2).
\end{equation*}
Finally, it follows that
\begin{equation*}
  2\lambda(2) \nu (f (-\mathcal{L}^*)f) \le  \nu (f (-\mathcal{L})f) \le 2\kappa \nu (f (-\mathcal{L}^*)f)
\end{equation*}
and therefore $2\lambda(2) \lambda^*(N) \le \lambda(N) \le 2\kappa \lambda^*(N)$.
In the same way, $2\lambda(2) \lambda^{*,loc}(N) \le \lambda^{loc}(N) \le 2\kappa \lambda^{*,loc}(N)$ is shown.

Now, it remains to show that $\kappa \le 1$. This follows from this simple inequality obtained by Schwarz's inequality:
\begin{align*}
\nu_{2,\omega} (f (-\mathcal{L}) f ) & = \frac{1}{8}  \nu_{2,\omega} ( \int_{-\pi}^{\pi} [ f(R^{12}_{\theta}\eta)-f(\eta)]^2 (\rho(\theta)+\rho(-\theta)) d\theta ) \\
& \le \frac{1}{4}  \nu_{2,\omega} ( \int_{-\pi}^{\pi} [f(R^{12}_{\theta}\eta)^2 +f(\eta)^2] (\rho(\theta)+\rho(-\theta)) d\theta ) = \nu_{2,\omega}(f^2).
\end{align*}
We use only here the assumption that $\rho(\theta)$ is a probability density on the circle. 

%***********************************************************************************************************************************

\section{General setting}

The general setting can be described as follows. We consider a product space $\Omega=X^N$ where $X$, the single component space is a measurable space equipped with a probability measure $\mu$. On $\Omega$, we consider the product measure $\mu^N$. Elements of $\Omega$ will be denoted by $\eta=(\eta_1,\eta_2, \cdots, \eta_N)$. Next, we take a measurable function $\xi: X \to \R^m$, for a given $m \ge 1$, and we define the probability measure $\nu=\nu_{N,\omega}$ on $\Omega$ as $\mu^N$ conditioned on the event 
\begin{equation*}
\Omega_{N,\omega}:=\{\eta \in \Omega; \sum_{i=1}^N \xi(\eta_i)=\omega\},
\end{equation*} 
where $\omega \in \Theta_N$ is a given parameter and $\Theta_N:=\{ \sum_{i=1}^N \xi(\eta_i); \eta \in X^N\}$ . We interpret the constraint on $\Omega_{N,\omega}$ as a conservation law.

In all the examples considered below there are no difficulties in defining the conditional probability $\nu$, therefore we do not attempt here at a justification of this setting in full generality but rather refer to the examples for full rigor. As pointed out in \cite{Ca08}, the crucial property of $\nu$ is that, for any set of indices $A$, conditioned on the $\sigma$-algebra $\mathcal{F}_A$ generated by variables $\eta_i, i \notin A$, $\nu$ becomes the $\mu$-product law over $\eta_j, \ j \in A$, conditioned on the event
\begin{equation*}
\sum_{j \in A} \xi(\eta_j)=\omega - \sum_{i \notin A} \xi(\eta_i).
\end{equation*} 

We introduce some notations in analogy with the last section. For $N \ge 3$, $\eta \in X^N$, $1 \le i <j \le N$ and $f: X^N \to \R$, define $f^{i,j}_{\eta} : X^2 \to \R$ as
\begin{equation*}
f^{i,j}_{\eta}(p,q)=f(\eta_1,\eta_2, \cdots, \eta_{i-1}, p, \eta_{i+1}, \cdots, \eta_{j-1}, q, \eta_{j+1}, \cdots, \eta_N).
\end{equation*}
For each $\omega \in \Theta_2$, fix a well defined (possibly unbounded, with dense domain denoted by $\D(\L_0)$) nonnegative self-adjoint operator $\mathcal{L}_0=\mathcal{L}_{0}^{\omega}$ defined on $L^2(\nu_{2,\omega})$ satisfying $\L_0 f=0$ if $f$ is a constant function. We are interested in the process on $\Omega_{N,\omega}$ described by the infinitesimal generator
\begin{equation}\label{eq:gene}
\mathcal{L}f(\eta)=\frac{1}{N}\sum_{i < j}\mathcal{L}_{i,j} f(\eta)
\end{equation}
where $\mathcal{L}_{i,j} f (\eta) = (\mathcal{L}_0 f^{i,j}_{\eta} ) (\eta_i,\eta_j)= (\mathcal{L}_0^{\xi(\eta_i)+\xi(\eta_j)} f^{i,j}_{\eta} ) (\eta_i,\eta_j) $. In all the examples considered below there are no difficulties to see that for each $\omega \in \Theta_N$ there exits a dense subset of $L^2(\nu_{N,\omega})$ denoted by $\D(\L)$ such that for all $f \in \D(\L)$, $\L f \in L^2(\nu_{N,\omega})$ is well defined, and $f^{i,j}_{\eta} \in \D(\L_0)$ for all $i<j$ and $\eta \in \Omega_{N,\omega}$. Moreover, by the construction, $\L$ is nonnegative self-adjoint operator on $\D(\L)$. As before, we refer to the examples for fully rigorous fomulations.

We also define the local version of the dynamics on $\Omega_{{\La_N},\omega}$ defined by
\begin{equation*}
\mathcal{L}^{loc}f(\eta)=\sum_{\substack{x,y \in \La_N \\ \|x-y\|=1}}\mathcal{L}_{x,y} f(\eta)
\end{equation*}
where $\mathcal{L}_{x,y} f (\eta) = (\mathcal{L}_0 f^{x,y}_{\eta} ) (\eta_x,\eta_y)$. Here $f^{x,y}_{\eta}$ is defined in the same way as $f^{i,j}_{\eta}$, and the sum runs over all unordered pairs $x, y \in \La_N$ satisfying $\|x-y\|=1$.

The spectral gap $\la(N,\omega)$ (resp. $\la^{loc}(N,\omega)$ ) is defined by (\ref{eq:sg}) with $L^2(\nu)$ replaced by $\D(\L)$ (resp. $\D(\L^{loc})$), and $-\mathcal{L}^*$ replaced by $-\mathcal{L}$ (resp.$-\mathcal{L}^{loc}$). As a convention, we may set $\la(N,\omega)=+\infty$ if $\omega$ is such that the measure $\nu$ becomes a Dirac delta. This convention shall apply also for $\la^{loc}(N,\omega)$, $\la^{*}(N,\omega)$ and $\la^{*,loc}(N,\omega)$ where the last two terms will be defined below.

To obtain lower and upper bounds on $\la(N,\omega)$ and $\la^{loc}(N,\omega)$, we consider a binary collision process given by simple averages, which was introduced by Caputo in \cite{Ca08}. This process is described by the infinitesimal generator
\begin{equation}\label{eq:gene*}
\mathcal{L^*}f(\eta)=\frac{1}{N}\sum_{b}\{\nu[f| \mathcal{F}_b]-f\}
\end{equation}
where the sum runs over all $\binom{N}{2}$ unordered pairs $b=\{i,j\}$ and $\nu[f| \mathcal{F}_b]$ is the $\nu$-conditional expectation of $f$ given the variables $\eta_k, \ k \notin b$. Setting, as before, $D_{i,j}=D_b=\nu[ \cdot | \mathcal{F}_b]-\text{Id}$. As usual, we refer to the examples for fully rigorous formulations. As in the last section, we also consider the local version of the process described by the infinitesimal generator
\begin{equation}\label{eq:gene*loc}
\mathcal{L}^{*,loc}f(\eta)=\sum_{\substack{x,y \in \La_N \\ \|x-y\|=1}}D_{x,y}f(\eta)
\end{equation}
where $\|x-y\|=\sum_{i=1}^d |x_i - y_i|$.
%Define $\la^{*,loc}(N,\omega)$ by the right hand of (\ref{eq:sg}) for $\mathcal{-L}^{*,loc}$ in place of $\mathcal{-L}$.
%In view of this symmetry and the ergodicity, the second smallest eigenvalue of $-\mathcal{L}$ on $\mathcal{S}_{\e,N}$ is given by
%\begin{equation}\label{eq:sg}
%\la^*(N,\omega):=\inf\Big\{ \frac{ \nu(f(\mathcal{-L^*})f)}{\nu(f^2)}  \Big| \nu(f)=0, \ f \in L^2(\nu) \Big\}.
%\end{equation}
%We call $\la^*$ the spectral gap of $- \mathcal{L^*}$. 
Note that in all the examples considered below, it is easy to check that $\D(\L^*)=\D(\L^{*,loc})=L^2(\nu)$. The spectral gap $\la^*(N,\omega)$ is defined by (\ref{eq:sg}) and $\la^{*,loc}(N,\omega)$ is defined by (\ref{eq:sg}) with $-\mathcal{L}^*$ replaced by $-\mathcal{L}^{*,loc}$. 

\begin{remark}
$\mathcal{L^*}$ can be seen as a special case of $\L$ in the form (\ref{eq:gene}) with $L_0^{\omega}f=\nu_{2,\omega}(f)-f$ for $f \in L^2(\nu_{2,\omega})$. 
\end{remark}

First, we show a comparison theorem between $\la^{*,loc}(N,\omega)$ and $\la^{*}(N,\omega)$.
\begin{theorem}\label{thm:nngeneral}
For any $N \ge 2$ and $\omega \in \Theta_N$,
  \begin{equation*}
\la^{*,loc}(N,\omega) \ge \frac{1}{96 d N^2}\la^*(|\La_N|,\omega).
  \end{equation*}
In particular,   
  \begin{equation}\label{eq:unibound*}
\inf_{N \ge 2} \inf_{\omega \in \Theta_N}\la^*(N,\omega) >0 
  \end{equation}
implies   
  \begin{equation}\label{eq:uniboundloc*}
  \inf_{N \ge 2} \inf_{\omega \in \Theta_{|\La_N|}} N^2 \la^{*,loc}(N,\omega) >0.
  \end{equation}
\end{theorem}
\begin{proof}
We repeat the proof of Theorem \ref{thm:nn}. Indeed we only used the property that the generators $\L^*$ and $\L^{*,loc}$ are described in the forms (\ref{eq:gene*}), (\ref{eq:gene*loc}) with the special operators $D_{i,j}$. 
\end{proof}

Now, we give a comparison theorem between $\la(N,\omega)$ (resp. $\la^{loc}(N,\omega)$) and $\la^{*}(N,\omega)$ (resp. $\la^{*,loc}(N,\omega)$). Define $\la(2)=\displaystyle \inf_{\omega \in \Theta_2}\la(2,\omega)$ and $\kappa$ as
\begin{equation*}
\kappa:=\sup_{\omega \in \Theta_2} \sup\Big\{ \frac{ \nu_{2,\omega}(f(-\mathcal{L})f)}{\nu_{2,\omega}(f^2)}  \Big| \nu_{2,\omega}(f)=0, \ f \in \D(\L) \Big\}
\end{equation*}
where $\L$ is the generator for $N=2$, namely $\frac{1}{2}\L_0$.
Here, as a convention, we may set $\sup\Big\{ \frac{ \nu_{2,\omega}(f(-\mathcal{L})f)}{\nu_{2,\omega}(f^2)}  \Big| \nu_{2,\omega}(f)=0, \ f \in \D(\L) \Big\}=-\infty$ if $\omega$ is such that the measure $\nu$ becomes a Dirac delta.

\begin{theorem}\label{thm:comparisongeneral}
For any $N \ge 2$ and $\omega \in \Theta_N$,
  \begin{align}
&  2\la(2) \la^*(N,\omega) \le \la(N, \omega)  \le 2 \kappa \la^*(N, \omega), \label{eq:comparisongeneral} \\
&  2\la(2) \la^{*,loc}(N,\omega) \le  \la^{loc}(N,\omega) \le 2 \kappa \la^{*,loc}(N, \omega). \label{eq:comparisongeneral2}
  \end{align}
In particular, if $\la(2) >0$, then (\ref{eq:unibound*}) implies 
  \begin{align}
& \inf_{N \ge 2} \inf_{\omega \in \Theta_N}\la(N,\omega) >0 \label{eq:unibound} \\
\text{and} \quad \quad  & \inf_{N \ge 2} \inf_{\omega \in \Theta_{|\La_N|}} N^2 \la^{loc}(N,\omega) >0. \label{eq:uniboundloc}
  \end{align}
On the other hand, if $\kappa < \infty$, then (\ref{eq:unibound}) implies (\ref{eq:unibound*}), (\ref{eq:uniboundloc*}) and (\ref{eq:uniboundloc}). 
\end{theorem}

\begin{proof}
We repeat the steps of the proofs of Theorem \ref{thm:comparison} and \ref{thm:comparisonupper} to show (\ref{eq:comparisongeneral}) and (\ref{eq:comparisongeneral2}). Indeed, this is a simple consequence of (\ref{eq:2-Nrelation}) with $\xi(\eta_i)+\xi(\eta_j)$ in place of $\eta_i^2+\eta_j^2$, and the fact that $\L=\frac{1}{2}\L_0$ for $N=2$, which always hold under our general setting. Note that since we assume that $\L_0f=0$ for any constant $f$, we have for any $\omega \in \Theta_2$ and $g \in L^2(\nu_{2,\omega})$,
\[
 \nu_{2, \omega}(  g (-\mathcal{L}_0) g  ) =  \nu_{2, \omega}(  \{  g - \nu_{2, \omega} (g) \}  (-\mathcal{L}_0) \{  g - \nu_{2, \omega} (g) \}   ) ,
\]
and therefore,
 \[
2 \lambda(2) \nu_{2, \omega}( \{  g - \nu_{2, \omega} (g) \} ^2 ) \le \nu_{2, \omega}(  g (-\mathcal{L}_0) g  ) \le 2 \kappa \nu_{2, \omega}( \{  g - \nu_{2, \omega} (g) \} ^2 ).
\] 

The latter part of the theorem follows from Theorem \ref{thm:nngeneral} and the former part of the theorem immediately, noting that (\ref{eq:unibound}) implies $\la(2)>0$.
\end{proof}

\begin{remark}
There exist many of models with the spectral gap satisfying $\la(2, \omega)>0$ for all $\omega \in \Theta_2$, but $\la(2)=0$. For these models, it is clear that the required lower bound (\ref{eq:unibound}) or (\ref{eq:uniboundloc}) does not hold. For these models, we should give the estimate of $\la(N, \omega )$ not only in terms of $N$ but also in $\omega$ (cf. \cite{NS}, \cite{S}).
\end{remark}
\begin{remark}
By definition, $\la^*(2, \omega)=\frac{1}{2}$ for all $\omega$ except for $\omega$ such that $\la^*(2, \omega)=\la(2, \omega)=+\infty$. Therefore, for $N=2$, (\ref{eq:comparisongeneral}) states that the following trivial relation holds:
  \begin{align*}
\la(2)  \le \la(2, \omega)  \le \kappa.
  \end{align*}
\end{remark}

\begin{theorem}\label{thm:main}
Assume $\displaystyle \la^*(3) := \inf_{\omega \in \Theta_3}\la^*(3,\omega) > \frac{1}{3}$ and $\la(2) >0$. Then, (\ref{eq:unibound*}), (\ref{eq:uniboundloc*}), (\ref{eq:unibound}) and (\ref{eq:uniboundloc}) hold.
\end{theorem}
\begin{proof}
Caputo proved in \cite{Ca08} that for $N \ge 2$ and $\omega \in \Theta_N$,
  \begin{equation*}
\la^*(N,\omega) \ge (3\la^*(3)-1)(1-\frac{2}{N})+\frac{1}{N}
  \end{equation*}
holds. Therefore, $\displaystyle \la^*(3) > \frac{1}{3}$ implies (\ref{eq:unibound*}) holds, and therefore (\ref{eq:uniboundloc*}) also holds by Theorem \ref{thm:nngeneral}. Then, since we assume $\la(2) >0$, (\ref{eq:unibound}) and (\ref{eq:uniboundloc}) also hold by Theorem \ref{thm:comparisongeneral}.
\end{proof}

\begin{remark}
Whether the condition $\displaystyle \la^*(3) > \frac{1}{3}$ (or (\ref{eq:unibound*})) holds or not depends only on the triplet $(X, \xi,\mu)$. Namely the analysis of the spectral gap of the process described by the infinitesimal generator of the form (\ref{eq:gene}) is reduced to the analysis of the property of the triplet, that is, the state space, the conservation law and the reversible measure, and the spectral gap of the same system for $N=2$.
\end{remark}

\begin{remark}
It is known that $\la^*(3) > \frac{1}{3}$ is not the necessary condition for (\ref{eq:unibound*}). Indeed, Caputo showed in \cite{Ca08} that $\displaystyle \la^*(4) := \inf_{\omega \in \R^m}\la^*(4,\omega) > \frac{1}{4}$ and $\la^*(3) > 0$ also implies (\ref{eq:unibound*}).
\end{remark}

\section{Examples}

\subsection{Kac walk}
The model discussed in the introduction can be seen as a special case of our general setting, so that Theorem \ref{thm:nn}, Theorem \ref{thm:comparison} and Theorem \ref{thm:comparisonupper} become special cases of Theorem \ref{thm:nngeneral} and Theorem \ref{thm:comparisongeneral}, respectively. Here $X=\R$, $\xi(\eta)=\eta^2$ (with $m=1$) and $\mu$ is the centered Gaussian measure with variance $v>0$. The choice of $v$ does not influence the determination of $\nu_{N, \omega}$. As shown in \cite{Ca08}, this model satisfies $\la^*(3) > \frac{1}{3}$ and therefore (\ref{eq:unibound*}) holds.

\subsection{Energy exchange model}
Here we consider a special class of the energy exchange models introduced in \cite{GKS} by Grigo et al. We refer \cite{GKS} for background and motivation on the model.
Let $X=\R_+$, $\xi(\eta)=\eta$ (with $m=1$), and $\mu$ be the Gamma distribution with a shape parameter $\gamma >0$ and a scale parameter $1$, i.e.
\[
\mu(d\eta)=\eta^{\gamma-1}\frac{e^{-\eta}}{\Gamma(\gamma)}d\eta.
\]
Note that the choice of the scale parameter does not influence the determination of $\nu_{N, \omega}$.
We consider a Markov process defined by its infinitesimal generator $\mathcal{L}^{loc}$, and $\mathcal{L}$ given by
\begin{equation*}
\mathcal{L}f(\eta)=\frac{1}{N}\sum_{i < j}\mathcal{L}_{i,j} f(\eta), \quad \quad
\mathcal{L}^{loc}f(\eta)=\sum_{\substack{x,y \in \La_N \\ \|x-y\|=1}}\mathcal{L}_{x,y} f(\eta)
\end{equation*}
where $\mathcal{L}_{i,j} f (\eta) = (\mathcal{L}_0 f^{i,j}_{\eta} ) (\eta_i,\eta_j)$, $\mathcal{L}_{x,y} f (\eta) = (\mathcal{L}_0 f^{x,y}_{\eta} ) (\eta_x,\eta_y)$, 
\begin{equation*}
\mathcal{L}_{0}f (\eta)=  \La(\eta_1,\eta_2) \int_{[0,1]} P(\eta_1,\eta_2,d\a)[f(T_{\a}\eta)-f(\eta)].
\end{equation*}
Here, $\La: \R^2_+ \to \R_+$ is a continuous function and $P(\eta_1,\eta_2,d\a)$ is a probability measure on
$[0,1]$, which depends continuously on $(\eta_1,\eta_2) \in \R^2_+$. The maps $T_{\a}$ model the energy exchange between two sites, and are defined by
\begin{equation*}
T_{\a}\eta=\eta+[\a \eta_{2} -(1-\a)\eta_1][\mathfrak{e}_1 - \mathfrak{e}_2] 
\end{equation*}
where $\mathfrak{e}_i$ denotes the $i$-th unit vector of $\R^2$. In words, the associated Markov process given by $\mathcal{L}^{loc}$ with $d=1$ goes as follows: Consider the one-dimensional lattice $\{1,2,...,N\}$. To every site $i$ of this lattice we associate an energy $\eta_i \in X=\R_+$. The collection of all the energies is denoted by $\eta = (\eta_1, \dots ,\eta_N) \in X^N$. To each nearest neighbor pair of the lattice we associate an independent exponential clock with a rate $\La$ that depends on the energies of this pair $\eta_i, \eta_{i+1}$. As soon as one of the $N-1$ clocks rings, say for the pair $(i,i+1)$, then a number $0 \le \a \le 1$ is drawn according to a distribution $P$, that only depends on the two energies $\eta_i, \eta_{i+1}$. Then, the updated configuration of the energies is such that the new energy at site $i$ is $\a(\eta_i +\eta_{i+1})$, the new energy at site $i+1$ is $(1-\a)(\eta_i + \eta_{i+1})$, and all other energies remain unchanged.

To guarantee the reversibility of the process with respect to $\mu^N$ (or $\mu^{\La_N}$), we assume the following:
\begin{Ass}\label{ass:1}
The rate function $\Lambda$ and the transition kernel $P$ are of the form
\begin{equation}\label{eq:productform}
  \begin{split}
\Lambda(\eta_1,\eta_2) & =\Lambda_s(\eta_1+ \eta_2)  \Lambda_r(\frac{\eta_1}{\eta_1+\eta_2}),  \\
 P(\eta_1,\eta_2,d\a) & =P(\frac{\eta_1}{\eta_1+\eta_2},d\a).
   \end{split}
\end{equation}
Moreover, $\Lambda_s(\sigma) \Lambda_r(\b) >0$ for all $\sigma>0$ and $0<\b<1$, $\sup_{0 < \b < 1} \La_r(\b) < \infty$, and the Markov chain on $[0,1]$ with transition kernel $P(\b,d\a)$ has a unique invariant distribution $p(\cdot)$ given by
\begin{equation*}
p(d\b)=d\b[\b(1-\b)]^{\gamma-1} \frac{\Gamma(2\gamma)}{\Gamma(\gamma)^2} \Lambda_r(\b) \frac{1}{Z}
\end{equation*}
where $Z$ is the normalizing constant, and $p$ is a reversible measure for the Markov chain generated by $P$.
\end{Ass}

\begin{remark}
Grigo et al pointed in \cite{GKS} that the representation
(\ref{eq:productform}) naturally occurs in models originating from mechanical systems. 
\end{remark}

Under Assumption \ref{ass:1}, Grigo et al showed in \cite{GKS} that $\mathcal{L}$ (resp. $\mathcal{L}^{loc}$) is reversible with respect to the product measure $\mu^N$ (resp. $\mu^{\La_N}$ ). Therefore, we define the spectral gap $\la(N,\omega)$ of $\mathcal{L}$ and $\la^{loc}(N,\omega)$ of $\mathcal{L}^{loc}$ for each $\omega >0$ as before.
%This form is exactly the one studied in Section 4 in \cite{GKS} and naturally occurs in models originating from mechanical systems. 
%Let, $\nu^N_{d,e}$ denote the product measure of $\nu_{d,e}$ on $\R_+^N$ and 
%$\pi_{e,N}$ denote the conditional probability measure of $\nu^N_{d,1}$ (or any $\nu^N_{d,e}$) on $\mathcal{S}_{\e,N}$. Throughout this paper, we assume that $\pi_{e,N}$ is a reversible measure for the process given by (\ref{eq:generator}); that is, each $\pi_{e,N}$ is invariant under the dynamics and detailed balance holds.

\begin{theorem}\label{thm:exam1}
If $\displaystyle \inf_{\sigma>0}  \Lambda_s(\sigma) >0$, then
\begin{align*}
\inf_{N \ge 2} \inf_{\omega >0 }\la(N,\omega) >0  \quad 
\text{and} \quad \inf_{N \ge 2} \inf_{\omega >0} N^2 \la^{loc}(N,\omega) >0. 
  \end{align*}
\end{theorem}

To prove the theorem, we first study the spectral gap for the generator $\mathcal{L}^*$ and $\mathcal{L}^{*,loc}$, which are the special case of the above model given by
\begin{equation*}
\Lambda^{*}_s(\sigma)=1, \quad \Lambda^{*}_r(\beta)=1, \quad P^{*}(\beta, d\a)=\frac{\Gamma(2\gamma)}{\Gamma(\gamma)^2} \{ \a(1-\a)\}^{d-1} d\a.
\end{equation*}
By definition, we can easily check that
\begin{equation*}
\mathcal{L}^*f(\eta)=\frac{1}{N}\sum_{i < j}D_{i,j} f(\eta), \quad \quad
\mathcal{L}^{*,loc}f(\eta)=\sum_{\substack{x,y \in \La_N \\ \|x-y\|=1}}D_{x,y} f(\eta),
\end{equation*}
and, by the unitary change of scale from $\Omega_{N, \omega}$ to $\Omega_{N,1}$, we have
\begin{equation*}
\lambda^{*}(N,\omega)=\lambda^{*}(N,1):=\lambda^*(N), \quad \lambda^{*,loc}(N,\omega)=\lambda^{*,loc}(N,1):=\lambda^{*,loc}(N).
\end{equation*}

To obtain the exact value of $\la^*(N)$ for $N \ge 2$, we recall Theorem 1.1 in \cite{Ca08}:
\begin{equation}\label{eq:caputo1}
\lambda^*(N) \ge (3\lambda^*(3)-1)(1-\frac{1}{N}) +\frac{2}{N}. 
\end{equation}
Moreover, if there exists $\psi: X \to \R$ such that the function
\begin{equation*}
f_3(\eta_1,\eta_2,\eta_3)=\sum_{i=1}^3 \psi(\eta_i)
\end{equation*}
satisfies, for $N=3$, $L^* f_3 =-\lambda^*(3)f_3 +$const., regardless of the value of $\displaystyle \sum_{i=1}^3 \eta_i$ (although the constant may depend on this value), then (\ref{eq:caputo1}) can be turned into an identity for each $N \ge 2$. 

Next, we apply the method introduced by Carlen, et al. in \cite{CCL} to solve the 3-dimensional problem. This approach was already used in \cite{Ca08} to show that $\lambda^*(N)=\frac{N+1}{3N}$ if $\gamma=1$.

%\begin{theorem}\label{thm:caputo}
%\begin{equation}\label{eq:caputo1}
%\lambda^*_{LR}(N) \ge (3\lambda^*_{LR}(3)-1)(1-\frac{1}{N}) +\frac{2}{N}. 
%\end{equation}
%Moreover, if there exists $\psi: X \to \R$ such that the function
%\begin{equation}
%f_3(x_1,x_2,x_3)=\sum_{i=1}^3 \psi(x_i)
%\end{equation}
%satisfies, for $N=3$, $L^*_{LR} f_3 =-\lambda^*_{LR}(3)f_3 +$const., regardless of the value of $\displaystyle \sum_{i=1}^3 x_i$ (although the constant may depend on this value), then (\ref{eq:caputo1}) can be turned into an identity for each $N \ge 2$. 
%\end{theorem}

\begin{theorem}\label{thm:non-degenerate}
For any $\gamma>0$, 
\begin{equation}\label{eq:gamnondeenerate}
\lambda^*(N)=\frac{\gamma N+1}{N(2\gamma+1)}. 
\end{equation}
\end{theorem}
\begin{proof}
As same way in examples 2.2 in \cite{Ca08}, we observe that when $N = 3$, then $\mathcal{L}^* + 1$ coincides with the average operator $P$ introduced in \cite{CCL}. Therefore we can apply the general analysis of Section 2 in \cite{CCL}. The outcome is that
\begin{equation}\label{eq:gap3}
\lambda^*(3) \ge \frac{1}{3}\min\{2+\mu_1,2-2\mu_2\}.
\end{equation}
where the parameters $\mu_1$ and $\mu_2$ are given by 
\begin{equation*}
\mu_1 = \inf_{\phi} \nu(\phi(\eta_1)\phi(\eta_2)), \qquad \mu_2 = \sup_{\phi} \nu(\phi(\eta_1)\phi(\eta_2))
\end{equation*}
with $\phi$ chosen among all functions $\phi:X \to \R$ satisfying $\nu(\phi(\eta_1)^2)=1$ and $\nu(\phi(\eta_1))=0$. Here $\nu$ stands for $\nu_{3,\omega}$ but we have removed the subscripts for simplicity. One checks that the parameters $\mu_1,\mu_2$ do not depend on $\omega$. Write $\mathcal{K}\phi(\zeta) = \nu [\phi(\eta_2)|\eta_1=\zeta], \zeta>0$. This defines a self-adjoint Markov operator on $L^2(\nu_1)$, where $\nu_1$ is the marginal on $\eta_1$ of $\nu$. In particular, the spectrum $Sp(\mathcal{K})$ of $\mathcal{K}$ contains $1$ (with eigen space given by the constants). Then $\mu_1$, $\mu_2$ are, respectively, the smallest and the largest value in $Sp(\mathcal{K}) \setminus \{1\}$, as we see by writing $ \nu(\phi(\eta_1)\phi(\eta_2))= \nu[\phi(\eta_1)\mathcal{K}\phi(\eta_1)]$. This is now a one-dimensional problem and $\mu_1, \mu_2$ can be computed as follows. To fix ideas we use the value $\omega= 1$ for the conservation law $\eta_1 + \eta_2 + \eta_3$. In this case $\nu_1$ is the law on $[0, 1]$ with density $\frac{\Gamma(3\gamma)}{\Gamma(2\gamma)\Gamma(\gamma)}{\eta}^{\gamma-1}(1-\eta)^{2\gamma-1}$. Moreover,
\begin{equation*}
\mathcal{K}\phi(\eta_1)=\frac{\Gamma(2\gamma)}{\Gamma(\gamma)^2(1-\eta_1)^{2 \gamma -1}}\int^{1-\eta_1}_0 \phi(\eta_2) \{\eta_2(1-\eta_1-\eta_2)\}^{\gamma-1} d\eta_2.
\end{equation*}
In particular, $\phi(\eta)=\eta-\frac{1}{3}$ is an eigenfunction of $\mathcal{K}$ with eigenvalue $-\frac{1}{2}$.
Moreover, $\mathcal{K}$ preserves the degree of polynomials so that if $Q_n$
denotes the space of all polynomials of degree $d \le n$ we have $\mathcal{K}Q_n \subset Q_n$. By
induction we see that for each $n \ge 1$ the polynomial $\zeta^n+q_{n-1}(\zeta)$, for a suitable
$q_{n-1} \in Q_{n-1}$, is an eigenfunction with eigenvalue $\mu_n =(-1)^n\frac{\Gamma(2 \gamma)\Gamma(n+\gamma)}{\Gamma(\gamma)\Gamma(n+2\gamma)}$, and it is orthogonal to $Q_{n-1}$ in $L^2(\nu_1)$. Since the union of $Q_n, n \ge 1$, is dense in $L^2(\nu_1)$ this shows that
there is a complete orthonormal set of eigenfunctions $\phi_n$, where $\phi_n$ is a polynomial
of degree $n$ with eigenvalue $\mu_n$ and $Sp(\mathcal{K}) = \{\mu_n , n = 0,1,\ldots \}$. Therefore we can
take $\mu_1=-\frac{1}{2}$ and $\mu_2 = \frac{1+\gamma}{2(1+2\gamma)}$ in the formula (\ref{eq:gap3}) and we conclude that $\lambda^*(3) \ge \frac{1+3 \gamma}{3(1+2\gamma)}$.

To end the proof, we take $f=\eta_1^2+\eta_2^2+\eta_3^2$ and, using $\nu[\eta_1^2| \eta_2]=\frac{1+\gamma}{2(1+2\gamma)} (\eta_2-1)^2$,
we compute
\begin{equation*}
\mathcal{L}^*f(\eta) = - \frac{1+3\gamma}{3(1+2\gamma)} f(\eta) + \textit{const}. 
\end{equation*}
Thus, $\lambda^*(3) = \frac{1+3\gamma}{3(1+2\gamma)}$. Clearly, the unitary change of scale does not alter the form of the eigenfunction so that (\ref{eq:gamnondeenerate}) follows.
\end{proof}

\begin{remark}
The consequence of Theorem \label{thm:non-degenerate} was shown in \cite{GF} with a different proof.  
\end{remark}

\begin{remark}
By Theorem 2.12 in \cite{GKS}, for $d=1$, $\lambda^{*,loc}(N) \ge \frac{\gamma}{2\gamma+1}\sin^2(\frac{\pi}{N+2})$ holds. However, to estimate the spectral gap with degenerate rate function, namely the case where $\displaystyle \inf_{\sigma>0}  \Lambda_s(\sigma) = 0$, we need to estimate the spectral gap on the complete graph (see \cite{S}). 
\end{remark}
%\subsection{Ginzburg-Landau model}
%$X=\R$, $\xi (\eta)=\eta$, $\mu$ is a small perturbation of the centered Gaussian measure. 

\begin{proof}[Proof of Theorem \ref{thm:exam1}]
By Theorem \ref{thm:comparisongeneral}, we only need to show that $\la(2)=\inf_{\omega >0} \la(2,\omega) >0$. By the assumption, for $N=2$, 
\begin{equation*}
\nu(f (-\mathcal{L})f)= \La_s(\omega)  \nu (\Lambda_r \Big(\frac{\eta_1}{\eta_1+\eta_2} \Big) \int_{[0,1]} \{ f(T_{\a}\eta)-f(\eta) \}^2 P\Big(\frac{\eta_1}{\eta_1+\eta_2},d\a\Big) ).
\end{equation*}
Therefore, by the unitary change of scale from $\Omega_{2, \omega}$ to $\Omega_{2,1}$, we have
\begin{equation*}
\la(2,\omega) = \frac{\La_s(\omega)}{\La_s(1)} \la(2,1).
\end{equation*}
Then, by our assumption, $\la(2,1)>0$ and therefore $\la(2)>0$.
\end{proof}

\subsection{Zero-range processes}
The class of zero-range processes is one of the well-studied interacting particle systems (cf. \cite{KL}). Though the process is of gradient type, the lower bound estimate of the spectral gap itself has been considered as an interesting problem and studied by several people (\cite{LSV}, \cite{M}, \cite{Ca}). 
Here, we take $X=\N \cup \{0 \}$, $\xi(\eta)=\eta$, and consider a partition function $Z(\cdot)$ on $\R_+$ by
\begin{equation*}
Z(\a)=\sum_{k \ge 0} \frac{\a^k}{g(1) g(2) \dots g(k)}
\end{equation*}
where $g : \N \to \R_+$ is a positive function. Let $\a^*$ denote the radius of convergence of $Z$:
\begin{equation*}
\a^* =\sup \{\a \in \R_+; Z(\a) < \infty\}.
\end{equation*}
In order to avoid degeneracy, we assume that the partition function $Z$ diverges
at the boundary of its domain of definition:
\begin{equation*}
\lim_{\a \uparrow \a^*} Z(\a) = \infty.
\end{equation*}
For $0 \le \a < \a^*$, let $p_{\a}$ be the probability measure on $X$ given by
\begin{equation*}
p_{\a}(\eta=k) =\frac{1}{Z(\a)} \frac{\a^k}{g(k)!}, \quad k \in X
\end{equation*}
where $g(k)!=g(1) g(2) \dots g(k)$. Note that the choice of $0 \le \a < \a^*$ does not influence the determination of $\nu=\nu_{N, \omega}$.

First, we consider $\mathcal{L}^*$ defined by (\ref{eq:gene*}) and study the value of $\la^*(3)$. Following the same argument of the computation of $\la^*(3)$ in Example 3.2, we can show that
\begin{equation*}
\lambda^*(3) \ge \frac{1}{3}\min\{2+\mu_1,2-2\mu_2\}
\end{equation*}
where the parameters $\mu_1$, $\mu_2$ are, respectively, the smallest and the largest value in $\{ Sp(\mathcal{K}_n) \setminus \{1\}; n \in \N \} $ and 
$\mathcal{K}_n = (\mathcal{K}^{(n)}_{ij})$ is the $n \times n$ matrix given by 
\begin{equation*}
\mathcal{K}^{(n)}_{ij}=
\begin{cases}
\frac{1}{g(n-j)!g(i-1-(n-j))!}(\sum_{l=0}^{i-1}\frac{1}{g(l)!g(i-1-l)!})^{-1} & \text{if} \quad i > n-j \\
0 & \text{if} \quad  i \le n-j.
\end{cases}
\end{equation*}
By the Perron-Frobenius theorem, $\mu_1 > -1$. Therefore, $\mu_2 < \frac{1}{2}$ is a sufficient condition for $\la^*(3) > \frac{1}{3}$. In \cite{ST}, the set $\{ Sp(\mathcal{K}_n) \setminus \{1\}; n \in \N \} $ is completely determined for the cases where $g(k)=1$ for all $k \in \N$ or $g(k)=k$ for all $k \in \N$. In the former case, $\mu_2=\frac{1}{3}$ and in the latter case $\mu_2=\frac{1}{4}$. It concludes that $\la^*(3) > \frac{1}{3}$ for both cases and therefore (\ref{eq:unibound*}) and (\ref{eq:uniboundloc*}) hold.

Next, we consider the generator of zero-range processes defined by 
\begin{equation*}
\mathcal{L}f(\eta)=\frac{1}{N}\sum_{i < j}\mathcal{L}_{i,j} f(\eta), \quad \quad
\mathcal{L}^{loc}f(\eta)=\sum_{\substack{x,y \in \La_N \\ \|x-y\|=1}}\mathcal{L}_{x,y} f(\eta)
\end{equation*}
where $\mathcal{L}_{i,j} f (\eta) = (\mathcal{L}_0 f^{i,j}_{\eta} ) (\eta_i,\eta_j)$, $\mathcal{L}_{x,y} f (\eta) = (\mathcal{L}_0 f^{x,y}_{\eta} ) (\eta_x,\eta_y)$, 
\begin{equation*}
\mathcal{L}_{0}f (\eta_1,\eta_2)=  g(\eta_1)\{ f(\eta_1 -1, \eta_2 +1 )-f(\eta_1,\eta_2) \}+ g(\eta_2)\{ f(\eta_1 +1, \eta_2 -1 )-f(\eta_1,\eta_2)\}.
\end{equation*}
As a convention, we take $g(0)=0$. To apply Theorem \ref{thm:main}, we need to study $\la(2)=\inf_{\omega \in \N_{\ge 0}}\la(2,\omega)$. For the choice $g(k)=1$ for all $k \in \N$, it is known that $\la(2)=0$. On the other hand, a sufficient condition for $\la(2)>0$ was given in \cite{LSV} as follows:
\begin{proposition}\label{prop:onesite}
Assume that the following two conditions are satisfied:\\
(i) $\sup_{k} |g(k+1)-g(k)| < \infty$, \\
(ii) There exists $k_0 \in \N$ and $C > 0$ such that $g(k)-g(j) \ge C$ for all $k \ge j + k_0$. \\
Then, we have $\la(2)>0$.
\end{proposition}
\begin{theorem}
Assume that the two conditions in Proposition \ref{prop:onesite} are satisfied, and $\mu_2 <\frac{1}{2}$. Then, (\ref{eq:unibound}) and (\ref{eq:uniboundloc}) hold.
\end{theorem}
\begin{proof}
By the above argument, we can apply Theorem \ref{thm:main} straightforwardly.
\end{proof}

\end{document}